\documentclass[12pt]{amsart}\usepackage{enumerate}
\usepackage{amsmath, amssymb, amsthm, amsfonts}
\usepackage{tikz}
\usepackage{geometry}
\usepackage{bbm} 
\usepackage{url}
\usepackage{cmap} 

\newtheorem{introthm}{Theorem}
\newtheorem{introcor}[introthm]{Corollary}
\newtheorem{theorem}{Theorem}[section]

\newtheorem{lemma}[theorem]{Lemma}

\newtheorem*{claim*}{Claim}
\newtheorem{assumption}[theorem]{Assumption}

\theoremstyle{definition}
\newtheorem{remark}[theorem]{Remark}
\newtheorem*{remark*}{Remark}

\newtheorem{definition}[theorem]{Definition}

\def\E{\mathbb{E}}
\def\P{\mathbb{P}}

\def\IR{\mathbb{R}}
\def\IZ{\mathbb{Z}}
\def\IW{\mathcal{W}}
\def\IA{\mathcal{A}}
\def\IM{\mathcal{M}}
\def\PS{\mathcal{P}}  

\def\eps{\varepsilon}
\def\ga{\gamma}
\def\Ga{\Gamma}

\def\sm{\setminus}

\DeclareMathOperator{\Aut}{Aut}
\DeclareMathOperator{\Stab}{Stab}
\DeclareMathOperator{\var}{var}
\DeclareMathOperator{\cov}{cov}

\DeclareMathOperator{\supp}{supp}

\newcommand{\defeq}{\mathrel{\vcenter{\baselineskip0.5ex \lineskiplimit0pt
                     \hbox{\scriptsize.}\hbox{\scriptsize.}}}%
                     =}

\def\ind{\mathbbm{1}} 

\begin{document}

\title{Entropy and expansion}

\author[Cs\'{o}ka]{Endre Cs\'{o}ka}
\address{MTA Alfr\'ed R\'enyi Institute of Mathematics, Budapest, Hungary} 
\email{csokaendre@gmail.com}

\author[Harangi]{Viktor Harangi}
\address{MTA Alfr\'ed R\'enyi Institute of Mathematics, Budapest, Hungary} 
\email{harangi@renyi.hu}

\author[Vir\'{a}g]{B\'{a}lint Vir\'{a}g}
\address{Department of Mathematics, University of Toronto}
\email{balint@math.toronto.edu}

\keywords{Entropy inequality, expansion, Cheeger constant, graph isoperimetry, 
factor-of-IID, local algorithm, independent set}

\subjclass[2010]{94A17, 60K35, 37A50, 05E18, 05C69}


%
\begin{abstract}
Shearer's inequality bounds the sum of joint entropies of random
variables in terms of the total joint entropy.  We give another lower
bound for the same sum in terms of the individual entropies when the
variables are functions of independent random seeds. The inequality
involves a constant characterizing the expansion properties of the
system.

Our results generalize to entropy inequalities used in recent work 
in invariant settings, including the edge-vertex inequality for
factor-of-IID processes, Bowen's entropy inequalities, and Bollob\'as's
entropy bounds in random regular graphs.

The proof method yields inequalities for other measures of randomness,
including covariance.

As an application, we give upper bounds for independent sets 
in both finite and infinite graphs.
\end{abstract}

\maketitle

\section{Introduction}

In recent years entropy inequalities for so-called \emph{local algorithms} 
and \emph{factor-of-IID processes} have been used to great effect in proving 
new results about random regular graphs and regular graphs of large girth. 
The method of entropy bounds actually goes back to a 1981 result of Bollob\'as \cite{bollobas} 
where he gave an upper bound on the independence ratio of random regular graphs. 
The counting argument behind that result essentially proves 
an \emph{edge-vertex entropy inequality} in a special setting. 
Interestingly, this entropy inequality also comes up 
in the seminal work of L.~Bowen on the $f$-invariant \cite{f-invariant,bowen}. 

All previous proofs for the edge-vertex entropy inequality heavily relied on the fact that 
the underlying graph does not contain (short) cycles. 
One would think that the acyclic property of the graph is not crucial. 
Indeed, in this paper we present an approach that directly relates the existence 
of such entropy inequalities to expansion properties of the graph 
allowing us to vastly generalize these inequalities. 

\subsection{The setup} 

Throughout the paper $X_v$, $v \in V$, will denote a finite or infinite system of random variables obtained as follows. 
Suppose that we have a collection of independent random variables $Z_1, Z_2, \ldots$ that we will refer to as \emph{random seeds}. 
Then for each $v$, $X_v$ is obtained as a measurable function of some specified subcollection of these random seeds. 
The main result is a general entropy inequality involving the Shannon entropies of our random variables $X_v$ 
and the joint entropies of $X_v$ for given subsets of $V$. 
We will see that the best coefficient for which our inequality holds 
can be expressed as a certain hyperedge expansion of the system. 

One may think of $X_v$, $v \in V$, as the output of a randomized (local) algorithm on a (large) network, 
where $V$ denotes the set of nodes. The algorithm is distributed among the nodes 
with each node having access only to a certain subcollection 
of the random seeds used by the algorithm, 
usually dictated by the restriction that each node 
can communicate only with close-by nodes. 
The entropy inequalities presented in this paper 
provide constraints for what can be achieved 
by such distributed algorithms. 

As we will see, the finite version nicely complements \emph{Shearer's inequality}, 
while the infinite version generalizes the afore-mentioned edge-vertex entropy inequality. 

\subsection{The finite version} 

For the moment let $X_v$, $v \in V$, denote any random variables indexed by elements of a finite set $V$ 
such that each $X_v$ takes values from a finite set. 
Suppose that we have a hypergraph $G$ over the vertex set $V$. 
This simply means that we have an edge set $E$ containing subsets $e$ of $V$, 
that is, each hyperedge $e \in E$ is a subset $e \subseteq V$. 
For the sake of simplicity the reader may think of $G$ 
as an ordinary graph where each edge $e$ is a pair of vertices. 
One can consider the joint entropy $H(X_e)$ of 
the random variables $X_v$, $v \in e$. 
Given a hypergraph over $V$ with edge set $E$, 
\emph{Shearer's inequality} \cite{shearer} says that  
\begin{equation} \label{eq:shearer}
\sum_{e \in E} H( X_e ) \geq k H(X_V) ,
\end{equation}
where $k$ is such that each $v \in V$ is contained by at least $k$ hyperedges 
(i.e.~$k$ is the minimum degree of the hypergraph) 
and $H(X_V)$ denotes the joint entropy of all $X_v$, $v \in V$. 
Shearer's inequality has many elegant applications in various areas, 
most notably in set systems and graph homomorphisms. 

We would like to complement \eqref{eq:shearer} by proving 
another lower bound for the same sum of joint entropies 
but this time in terms of the individual entropies $H(X_v)$. 
If no additional assumption is made, 
then only trivial bounds can be obtained 
(i.e.~bounds that follow from the inequalities $H(X_e) \geq H(x_v)$, $v \in e$). 

To get a non-trivial inequality, we will work with random variables $X_v$ 
that are measurable functions of independent random seeds with the sole restriction 
that each $X_v$ ``can see'' only some specified subcollection of the seeds. 
We will prove that 
$$ \sum_{e \in E} H( X_e ) \geq \beta \sum_{v \in V} H( X_v ) ,$$ 
where $\beta$ is related to the expansion properties of the hypergraph 
with respect to sets describing which $X_v$ can see a given seed; 
see Theorem \ref{thm:hypergraph_version}. 
To be more specific, but without having to go into too many details, 
we state our inequality here for the case when $G$ is a regular graph. 
\begin{introthm}[Regular graph version] \label{thm:regular} 
Suppose that $G$ is a $d$-regular graph with vertex set $V$ and edge set $E$. 
Suppose that $Z_1, \ldots, Z_m$ are independent random variables 
and a subset $W_i \subseteq V$ is given for each $i=1,\ldots, m$. 
Informally, the random value of $Z_i$ is passed on to the vertices in $W_i$, 
and for each $v \in V$ we have a random variable $X_v$ 
that is an arbitrary measurable function of the values passed on to $v$. 
More precisely, for each vertex $v \in V$ let $X_v$ be a measurable function of 
those $Z_i$ for which $v \in W_i$. 
Then 
\begin{equation*}
\sum_{\{u,v\} \in E} H(X_u, X_v) \geq \beta \sum_{v \in V} H(X_v) 
\mbox{, where } \beta = \frac{1}{2} \left( d + 
\min_{1 \leq i \leq m; W \subseteq W_i} \frac{| \partial W |}{|W|} \right) .
\end{equation*}
Here $\partial W$ denotes the \emph{edge boundary} of $W$, that is, 
the set of edges with one endpoint in $W$ and one endpoint outside $W$. 
The above minimum of $|\partial W| / |W|$ is 
the \emph{edge Cheeger constant} (or \emph{isoperimetry number}) 
of $G$ with respesct to the set system $W_1, \ldots, W_m$. 
\end{introthm}
We will see that this $\beta$ is the largest coefficient for which the theorem holds. 

\subsection{Factor-of-IID processes} \label{sec:intro_fiid}

The following question arises: can we use our method to say something 
if not finitely but countably many random variables $X_v$ are given? 
In this case the sum of all entropies may be infinite. 
Therefore, to get a meaningful statement, one needs to assume some kind of invariance. 
A natural framework to consider is \emph{factor-of-IID processes}. 

Factors of IID are closely related to randomized local algorithms. 
They are also extensively studied by ergodic theory 
under the name of \emph{factors of Bernoulli shifts}. 
In short, we start with independent and identically distributed (say $[0,1]$ uniform) 
random variables $Z_v$, $v \in V$ (i.e.~the \emph{IID process}), 
then we apply a measurable function $F \colon [0,1]^V \to M^V$ 
that is required to be $\Ga$-equivariant for some group $\Ga$ acting on $V$.  
The collection of random variables $X_v$, $v \in V$, we obtain this way 
is called a \emph{factor-of-IID process}. 
Note that the joint distribution of $X_v$, $v \in V$, is $\Ga$-invariant. 
In this paper the state space $M$ will usually be finite.

An important special case, mainly because of its connection to 
random regular graphs and other large-girth regular graphs, 
is when $V$ is the vertex set of the $d$-regular tree $T_d$ 
equipped with the natural action of the automorpihsm group $\Ga = \Aut(T_d)$. 
In this setting various entropy inequalities 
have been known for factor-of-IID processes. 
The simplest one is the following edge-vertex inequality: 
\begin{equation} \label{eq:edge-vertex}
H( X_u, X_v ) \geq \frac{2(d-1)}{d} H( X_v ) \mbox{, where $(u,v)$ is an edge.}
\end{equation}
This inequality played a central role 
in a couple of intriguing results recently \cite{ev,mustazeebalint}. 
All previous proofs of \eqref{eq:edge-vertex} are based on counting arguments 
for random regular graphs or for random permutations, 
and as such they heavily build on the acyclic nature of $T_d$, 
see \cite{invtree,bowen,mustazee}. 
Further inequalities and generalizations can be found in \cite{mut_inf,entr_ineq}. 

The new approach presented in this paper 
does not assume the acyclic property of the system (only certain local expansion). 
It provides a clean new proof for \eqref{eq:edge-vertex} 
and leads to numerous new versions and generalizations. 
Once again, we want to avoid too many technicalities at this point 
so we do not state our result in its most general form. 
Theorem \ref{thm:fiid} below is concerned with transitive graphs $G$ 
satisfying some additional invariance called \emph{unimodularity}. 

There are several equivalent definitions for unimodular transitive graphs. 
The one most convenient for applications is the so-called \emph{Mass-Transport Principle} saying that 
$\sum_{v \in V} h(o,v) = \sum_{v \in V} h(v,o)$ holds for any fixed vertex $o$ of the vertex set $V$ 
and for any function $h \colon V \times V \to [0,\infty)$ that is diagonally $\Aut(G)$-invariant 
(i.e.~$h(u,v) = h(\varphi(u),\varphi(v))$ $\forall \varphi \in \Aut(G)$ $\forall u,v \in V$). 
For an equivalent definition that is easier to check, 
let $m_{u,v}$ denote the size of the orbit of $v$ under the stabilizer of $u$, that is, 
the size of the set $\{ \varphi(v) \, : \, \varphi \in \Aut(G) \, ; \, \varphi(u) = u \}$. 
Then a transitive $G$ is unimodular if and only if $m_{u,v} = m_{v,u}$ for any pair of vertices.
\begin{introthm} \label{thm:fiid}
Let $G$ be a unimodular transitive graph with vertex set $V$ and with degree $d$. 
Suppose that $X_v$, $v \in V$, is an $\Aut(G)$-factor-of-IID process 
with a finite state space. Then for an arbitrary $o \in V$ we have  
$$ \frac{1}{2} \sum_{o' \sim o} H( X_o, X_{o'} ) \geq \beta H( X_o ) \mbox{, where } 
\beta = \frac{1}{2} \left( d + \inf_{W \subset V; |W|<\infty} \frac{ | \partial W | }{|W|} \right) .$$
Here $o' \sim o$ means that $o'$ and $o$ are connected by an edge in $G$, 
that is, $o'$ runs through all neighbors of the fixed vertex $o$. 
Note that the infimum of $|\partial W| / |W|$ over finite sets $W$ 
is usually called the \emph{edge Cheeger constant} of $G$. 
\end{introthm}
In particular, for $G=T_d$ the Cheeger constant is $d-2$ 
and we get back the original edge-vertex inequality \eqref{eq:edge-vertex}. 
See Theorem \ref{thm:general} for our general result, 
which is concerned with quasi-transitive group actions 
and joint entropies for arbitrary finite subsets. 

\subsection{Proof method}
In the finite case our strategy is to reveal the independent seeds $Z_1, Z_2, \ldots, Z_m$ one by one 
and study how much information is gained about $X_e$ and $X_v$ at each step. 
After step $i$, the conditional entropy $H\left( X_e~|~Z_1,\ldots,Z_i \right)$ 
indicates the expected ``amount of uncertainty'' remaining. 
So the information gained at step $i$ can be defined as the difference 
$$ H\left( X_e~|~Z_1,\ldots,Z_{i-1} \right) - H\left( X_e~|~Z_1,\ldots,Z_i \right) .$$ 
It can be seen that this difference must be at least as much as 
the same difference for $X_v$ provided that the vertex $v$ is incident to $e$. 
This observation, combined with the fact that 
the information gain at step $i$ is zero for vertices outside $W_i$, 
leads to Theorem \ref{thm:regular}. 

As for the infinite setting, one needs to reveal the seeds 
in a random order and take expectation. This ensures that 
we have the required invariance to use the Mass-Transport Principle 
in the argument.

In fact, the approach described above goes beyond entropy. 
There are other ways of ``measuring the uncertainty'' in random variables. 
If such an uncertainty function satisfies a certain convexity criterion, 
then we can simply use this function---instead of entropy---in our arguments 
and obtain further inequalities of similar flavor. For example, we can get 
a bound for the correlation of neighbors in terms of the edge expansion.
We will discuss this in detail in Section \ref{sec:beyond}.  

Telescoping sums of conditional entropies, 
in particular of those in random order,
have been used in the literature before. 
The recent notion of \emph{percolative entropy} 
is also based on this idea, see \cite{austin_podder,seward}. 

\subsection{Applications}
As we have mentioned, these inequalities can be used to get constraints 
for what can be achieved by factor-of-IID processes or randomized local algorithms. 
They can be applied to various problems; 
in this paper we will focus on questions regarding independent sets.

Upper bounds have been known for the density of factor-of-IID independent sets 
in Cayley graphs. An infinite version of the Hoffman bound provides a non-trivial 
bound for non-amenable graphs \cite[Proposition 3.3]{lyons_nazarov}. 
The same paper explores possible improvements as well \cite[Section 4]{lyons_nazarov}.
Here, instead of using the spectral radius or other spectral parameters, 
we give a bound in terms of the Cheeger constant.
\begin{introthm} \label{thm:indset_bound}
Suppose that $G$ is a Cayley graph (or more generally, a unimodular transitive graph) 
with vertex set $V$ and degree $d$. 
Let $\tau$ denote the normalized edge Cheeger constant: 
$$ \tau = \inf_{W \subset V; |W|<\infty} \frac{ | \partial W | }{d|W|} .$$ 
Then the density of any factor-of-IID independent set in $G$ 
is at most $\varphi^{-1}( \tau )$, where $\varphi^{-1}$ is the inverse of the 
following function: 
$$ \varphi(q) = \frac{ - 2q\log q - (1-2q)\log(1-2q) }{ - q\log q - (1-q)\log(1-q) } - 1 
\mbox{ for } q \in (0,1/2) .$$ 
It is easy to see that $\varphi$ continuously and monotone decreasingly maps 
$(0,1/2)$ onto $(0,1)$, and hence the inverse exists. 
\end{introthm}
It is, of course, very hard to assess the strength of this result 
without further details or examination. 
We will discuss the behaviour of $\varphi^{-1}$ in Section \ref{sec:appl}. 
Note, however, that $G$ is non-amenable if and only if $\tau > 0$, 
and the obtained bound $\varphi^{-1}( \tau )$ is strictly less than $1/2$ in that case. 
If $G$ is the $d$-regular tree $T_d$, then we get back the original Bollob\'as bound: 
$(2 \log d)/d$ asymptotically as $d \to \infty$. The exact value of the Cheeger constant
is also known for regular tessellations of the hyperbolic plane 
\cite{haggstrom_jonasson_lyons,higuchi_tomoyuki} 
so Theorem \ref{thm:indset_bound} can be applied to obtain good quantitative bounds 
for these lattices. Asymptotically we get the following.
\begin{introcor} \label{cor:tessellation}
Consider the regular tessellation of the hyperbolic plane 
by polygons with $k$ sides and with $d$ polygons at each vertex 
for some positive integers $k$ and $d$ satisfying $(k-2)(d-2) > 4$. 
Let $G_{d,k}$ denote the corresponding planar graph, that is, 
each face is a $k$-cycle and each degree is $d$. 
For any fixed $k$, the density of a factor-of-IID 
independent set in $G_{d,k}$ is at most 
$$ \left( 2 + \frac{2}{k-2} + o(1) \right) \frac{\log d}{d} \mbox{ as } d \to \infty .$$ 
\end{introcor}
To compare this to the spectral approach, we note that 
the Hoffman bound is, at best, of order $1/\sqrt{d}$ for any $d$-regular infinite graph, 
which is significantly weaker than the result above. 
In general, we cannot say that our bound (given by Theorem \ref{thm:indset_bound}) 
is always better: the Cheeger inequalities relate the spectral radius to the Cheeger constant, 
and these inequalities do not exclude the possibility that the spectral approach 
gives slightly better results but this is probably not the typical case. 
Moreover, we rarely have exact values or even good estimates for the spectral radius of 
the adjacency operator. It tends to be a somewhat easier task to find the Cheeger constant. 

Finally, we mention another application that shows that we can get results 
even in amenable graphs using local (or small set) expansion.
\begin{introthm} \label{thm:Zn}
Let $X$ be a factor-of-IID process with finite radius $R$ 
over the integer lattice $\IZ^n$ taking values $0$ or $1$. 
If $X$ defines an independent set (i.e.\ the values of two neighboring vertices cannot be both $1$), 
then the density of this independent set (i.e.\ the probability that a given vertex takes value $1$) 
is bounded above by 
$$ \frac{1}{2} - \frac{cn}{R\log(R/n)} $$
provided that $R>Cn$, where $c$ and $C$ are absolute constants. 
\end{introthm}

\subsection*{Outline of the paper}
In Section \ref{sec:2} we present and prove our entropy inequalities for a finite system of random variables, 
while in Section \ref{sec:3} we deal with the infinite case under the unimodularity condition. 
Section \ref{sec:beyond} explores possible generalizations, while Section \ref{sec:appl} discusses applications. 
For the reader's convenience we also include an Appendix (Section \ref{sec:app}) 
with some definitions and lemmas regarding entropy and conditional entropy.

\section{Finite setting} \label{sec:2}

\subsection{Expansion for hypergraphs}
Let $V$ be a finite set and let $G$ be a hypergraph over $V$, 
where the set of hyperedges is denoted by $E$. 
In other words, each hyperedge $e \in E$ is a subset of $V$. 
First we define a certain hyperedge expansion that will appear 
as a coefficient in our entropy inequality. 
\begin{definition}
Given a set $W \subset V$, we define its \emph{closure} $\overline{W}$ as the set of hyperedges incident to $W$:
$$ \overline{W} \defeq \left\{ e \in E \, : \, e \cap W \neq \emptyset \right\} .$$ 
Then the \emph{hyperedge expansion} of $G$ with respect to the sets $W_1, \ldots, W_m \subset V$ is defined as 
\begin{equation*} 
\beta \defeq \min_{1 \leq i \leq m; W \subseteq W_i} \frac{| \overline{W} |}{|W|} . 
\end{equation*}
More generally, we will consider hypergraphs with a nonnegative edge weight assigned to each $e \in E$, 
in which case we will need to replace $| \overline{W} |$ with the sum of weights in $\overline{W}$. 
\end{definition}
\begin{remark} \label{rm:expansions}
Note that the most standard edge expansion definition for graphs 
uses the \emph{edge boundary} $\partial W$ consisting of edges 
with one endpoint in $W$ and one endpoint in $V \sm W$. 
If $G$ is a $d$-regular graph, then the two notions are essentially the same: 
$$ d |W| = 2\left( \# \mbox{edges inside } W \right) + | \partial W | \mbox{ and } 
| \overline{W} | =\left( \# \mbox{edges inside } W \right) + | \partial W | ,$$
and hence 
$$ \beta = \min_{1 \leq i \leq m; W \subseteq W_i} \frac{ d|W| + | \partial W | }{2|W|} 
= \frac{1}{2} \left( d + \min_{1 \leq i \leq m; W \subseteq W_i} \frac{| \partial W |}{|W|} \right) .$$
\end{remark}

\subsection{An entropy inequality for finitely many random variables}
Let $V$ be a finite set and $X_v$, $v \in V$, be random variables, 
each $X_v$ taking finitely many values. 
For any subset $U \subseteq V$ one can consider the joint entropy of 
the random variables $X_v$, $v \in U$. This is simply defined as the Shannon entropy 
of the joint distribution of $X_U \defeq \left( X_v \right)_{v \in U}$ 
and will be denoted by $H(X_U)$. 
Usually we will think of such a subset $U$ as a hyperedge and use the notation $e \subseteq V$. 

Throughout this section we will work under the following assumption.
\begin{assumption} \label{assumption}
Let $Z_1, \ldots, Z_m$ be independent random variables. 
For each $Z_i$ a subset $W_i \subseteq V$ is given, 
and for every $v \in V$ we assume that $X_v$ 
is a measurable function of $Z_i, v \in W_i$. 
\end{assumption}
An equivalent formulation would be the following. 
Let $\IW$ denote the set of all subsets of all $W_i$. 
Then given countably many independent random bits, 
each bit is sent to $X_v$, $v \in W$, for some $W \in \IW$, 
and every $X_v$ is a measurable function of the bits received. 

For random variables $X_v$ obtained as above, 
the next theorem bounds any given sum of joint entropies in terms of the sum of the individual entropies. 
Note that Shearer's inequality \eqref{eq:shearer} provides a lower bound for the same sum 
in terms of the joint entropy of all $X_v$. 
\begin{theorem}[Hypergraph version] \label{thm:hypergraph_version} 
Under Assumption \ref{assumption} the following inequality is true for any set $E$ of hyperedges over $V$: 
\begin{equation*}
\sum_{e \in E} H(X_e) \geq \beta \sum_{v \in V} H(X_v) 
\mbox{, where } \beta = \min_{1 \leq i \leq m; W \subseteq W_i} 
\frac{\left| \overline{W} \right|}{|W|} .
\end{equation*}
Recall that $\overline{W}$ is defined as $\{ e \in E : e \cap W \neq \emptyset \}$.
\end{theorem}
When our hypergraph is a $d$-regular graph, 
then the coefficient $\beta$ can be expressed using the standard edge expansion 
with respect to subsets of $W_i$ (see Remark \ref{rm:expansions}) 
so in that case we get back Theorem \ref{thm:regular} stated in the introduction. 

Also note that in the $d$-regular case $\beta = d/2$ would be a trivial lower bound 
following from the fact that $H(X_u,X_v) \geq ( H(X_u) + H(X_v) )/2$. 
Similarly, a trivial upper bound ($\beta=d$) can be deduced using $H(X_u,X_v) \leq H(X_u) + H(X_v)$. 

Finally, we state and prove the most general version of our inequality 
where nonnegative weights are assigned to the hyperedges. 
In fact, we will assign a weight to each subset $U \subseteq V$, 
thinking of zero weights as non-edges. 
\begin{theorem}[Weighted hypergraph version] \label{thm:weighted_hg}
Suppose that to each subset $U \subseteq V$ 
a nonnegative weight $\alpha_U \in [0, \infty)$ is assigned. 
Then, under Assumption \ref{assumption}, the following inequality holds: 
\begin{equation} \label{eq:weighted}
\sum_{U \subseteq V} \alpha_U H(X_U) \geq \beta \sum_{v \in V} H(X_v) 
\mbox{, where } \beta = \min_{1 \leq i \leq m; W \subseteq W_i} 
\frac{\sum_{U \subseteq V; U \cap W \neq \emptyset} \alpha_U}{|W|} .
\end{equation}
Moreover, the above inequality is sharp: $\beta$ is the largest coefficient 
for which \eqref{eq:weighted} holds for any random variables satisfying Assumption \ref{assumption}. 
\end{theorem}
\begin{proof}
The idea is to reveal $Z_1,\ldots,Z_m$ one by one and 
look at how much information is gained on average about $X_U$ at each step. 
To be more precise, for any $i$ and $U$ the following conditional entropy can be defined:
$$ H\left( X_U~|~Z_1,\ldots,Z_i \right) .$$
(See the Appendix for basic definitions regarding Shannon entropy 
and for how the above conditional entropy can be defined if $Z_i$ is not discrete.) 
For $i=0$ this is simply $H(X_U)$, while for $i=m$ we get $0$ 
because $Z_1, \ldots, Z_m$ fully determine $X_U$. 
Then the average information gained at step $i$ can be defined as: 
$$ h(i,U) \defeq 
H\left( X_U~|~Z_1,\ldots,Z_{i-1} \right) - H\left( X_U~|~Z_1,\ldots,Z_i \right) .$$ 
In the case of one-element sets $U=\{v\}$ we will simply write $h(i,v)$. 

If we add up all $h(i,U)$ for a given $U$, 
then we get a telescoping sum and obtain that  
$$ H(X_U) = \sum_{i=1}^m h(i,U) .$$ 
We will also need the following observation: 
whenever $U' \subset U$ we have $h(i,U') \leq h(i,U)$ for each $i$. 
This is clear heuristically since for larger $U$
we should gain more information about $X_U$ at each step. 
For a rigorous proof see Lemma \ref{lem:info_gained} in the Appendix. 
In particular, it holds for any given $i$ and $U$ that 
$$ h(i,U) \geq \max_{v \in U} h(i,v) .$$ 
It follows that 
$$ \sum_{U \subseteq V} \alpha_U H(X_U) \geq 
\sum_{i=1}^m \sum_{U \subseteq V} \alpha_U \max_{v \in U} h(i,v) .$$
As for the right-hand side of the inequality \eqref{eq:weighted}, we have 
$$ \beta \sum_{v \in V} H(X_v) = \beta \sum_{i=1}^m \sum_{v \in V} h(i,v) .$$
Therefore it suffices to show that it holds for each $i$ that 
$$ \sum_{U \subseteq V} \alpha_U \max_{v \in U} h(i,v) \geq \beta \sum_{v \in V} h(i,v). $$
Note that, for fixed $i$, $h(i,v)$ is 
a $V \to [0, \infty)$ function that is supported on $W_i$. 
Hence the following claim completes the proof.
\begin{claim*}
For any function $f \colon V \to [0,\infty)$ it holds that 
$$
\frac{ \sum_{U \subseteq V} \alpha_U \max_{v \in U} f(v) }{ \sum_{v \in V} f(v) } 
\geq \min_{ W \subseteq \supp{f} } 
\frac{\sum_{U \subseteq V; U \cap W \neq \emptyset} \alpha_U}{|W|} . 
$$
In other words, the fraction on the left-hand side is minimized by 
an indicator function $\ind_W$. 
\end{claim*}
To prove this claim we take the following unique decomposition of $f$ 
into the sum of indicator functions: 
there exist sets $ \supp{f} = A_1 \supsetneq A_2 \supsetneq \cdots \supsetneq A_k $ 
and positive real numbers $c_1, \ldots, c_k > 0$ such that 
$$ f = \sum_{j=1}^k c_j \ind_{A_j} .$$ 
For such a decomposition one can always interchange taking maximum over a set $U$ with taking sum: 
$$ \max_{v \in U} f(v) = \sum_{j=1}^k c_j \max_{U} \ind_{A_j} = 
\sum_{j=1}^k c_j \ind_{U \cap A_j \neq \emptyset} .$$
Now considering the fraction in the claim, 
the total contribution of the $j$-th term 
to the numerator and denominator 
is $c_j \sum_{U \subseteq V; U \cap A_j \neq \emptyset} \alpha_U$ 
and $c_j |A_j|$, respectively. 
The claim follows by using that 
$\sum a_j / \sum b_j \geq \min_j a_j/b_j$ for any $a_j \geq 0$, $b_j>0$.

There is a slightly different way to finish the proof 
after having defined the values $h(i,U)$ for a given process $X$. 
We include this ending, too, partly because it is a nice argument 
and partly because it easily implies that our inequality is sharp. 
For the sake of simplicity we first present the proof 
under the additional assumption that 
$h(i,v)$ is an integer multiple of $\log 2$ for each $i$ and $v$. 
First we define a new set of random variables: 
let $\tilde{Z}_i$ consist of a finite but sufficiently large number of independent ``bits'' 
(taking values $0$ and $1$ with probability $1/2$ each), and 
let $\tilde{X}_v$ be the collection of the first $h(i,v)/\log 2$ bits 
of $\tilde{Z}_i$ for each $i=1,\ldots,m$. It is clear that 
$$ H( \tilde{X}_v ) = \sum_{i=1}^m h(i,v) = H( X_v ) $$ 
and 
$$ H( \tilde{X}_U ) = \sum_{i=1}^m \max_{v \in U} h(i,v) \leq \sum_{i=1}^m h(i,U) = H( X_U ). $$
It follows that it suffices to prove \eqref{eq:weighted} for this new set of random variables $\tilde{X}_v$. 
Now let us consider an arbitrary bit of $\tilde{Z}_i$. 
This bit is present in the random variables $\tilde{X}_v$, $v \in W$, for some subset $W \subseteq W_i$. 
Therefore its entropy ($\log 2$) will appear 
$\sum_{U \subset V; U \cap W \neq \emptyset} \alpha_U$ times on the left-hand side 
and $\beta |W|$ times on the right-hand side of \eqref{eq:weighted}. 
The former will be clearly larger than or equal to the latter 
provided that we choose $\beta$ as in the theorem.

If the values $h(i,v)/\log 2$ are not necessarily integers, 
then we need to slightly modify the above argument. 
Let $\tilde{Z}_i$ be a sequence of independent discrete random variables 
with prescribed entropies. We can clearly choose these entropies in such a way 
that for any $v \in V$ the value $h(i,v)$ can be obtained 
as the sum of the first $k(i,v)$ of these entropies for some nonnegative integer $k(i,v)$. 
Then $\tilde{X}_v$ can be defined as the collection of the first $k(i,v)$ elements 
of $\tilde{Z}_i$, $i=1, \ldots, m$. The rest of the proof remains essentially the same.

To see sharpness, take a subset $W \subseteq W_i$ with 
$\big( \sum_{U \subseteq V; U \cap W \neq \emptyset} \alpha_U \big) / |W| = \beta$, 
that is, $W$ has ``minimal expansion''. 
Let $X_v$ be the same measurable function of $Z_i$ for all $v \in V$, 
and let $X_v$ be almost surely constant for $v \not\in W$. It is easy to see that \eqref{eq:weighted} 
holds with inequality in this case. Usually there are many $W$ with minimal expansion 
in which case one can consider combinations of these examples for different $W$. 
(Much as in the construction of $\tilde{X}_v$ above.) 
\end{proof}

\subsection{The general problem}

All the inequalities proved in this section, and Shearer's inequality as well,  
fit into the following general problem. Two weighted sums of joint entropies 
of finitely many random variables are given along with sets $W_i$ 
describing which variables can use the same random seed. 
The task is to find the strongest inequality (the one with the best coefficient) between the two sums.
Solving the problem in full generality is probably too much to ask for, 
but there might exist a common generalization of \eqref{eq:weighted} and Shearer's inequality, 
which would already be very interesting.

In fact, it is not crucial for our approach to have individual entropies 
on the right-hand side of the inequality. 
The key observation in our proof was that the information gain is always larger for larger sets, 
that is, $h(i,U') \leq h(i,U)$ provided that $U' \subset U$. 
Note that we only used this fact in the case when $U'$ was a one-element set $\{v\}$. 
If we make full use of this fact, more complicated inequalities can be obtained. 
For example, one could deduce an inequality between the entropies corresponding to the stars 
in a graph and the entropies corresponding to the edges. By a star we mean a vertex $v$ and 
its neighbors $N(v)$. Given a simple graph $G=(V,E)$, a straightforward modification of the 
proof above yields the following star-edge entropy inequality:
\begin{equation}
\sum_{v \in V} H(X_{\{v\} \cup N(v)}) \geq \beta \sum_{e \in E} H(X_e) 
\mbox{, where } \beta = \min_{1 \leq i \leq m; F \subseteq \overline{W_i}} 
\frac{\left| \left\{ v : \exists e \in F \mbox{ s.t. } v \in e \right\} \right|}{|F|} .
\end{equation}
An infinite version of this inequality for factor-of-IID processes on $T_d$ says that 
the entropy of a star is at least $d/2$ times the entropy of an edge 
\cite{invtree,ev}.

\section{Unimodular setting} \label{sec:3}

In this section we investigate what can be said 
for infinitely many random variables. 
It turns out that our approach can be adapted 
provided that $X_v$, $v \in V$, are obtained from the seeds  
in a certain invariant way.
We start with the simplest setup. 

\subsection{Transitive graphs}
Let $G$ be a connected infinite graph that is transitive 
meaning that its automorphism group $\Aut(G)$ acts transitively on its vertex set $V$. 
First we recall the definition of factor-of-IID processes on $G$. 
We start with an IID process over $V$: this is simply a collection of independent 
random variables $Z_v$, $v \in V$, each uniformly distributed on $[0,1]$. 
Then we apply a measurable $\Aut(G)$-equivariant $[0,1]^V \to M^V$ mapping. 
The result is a collection of random variables $X_v$, $v \in V$, each taking values in $M$. 
A process that can be obtained this way is called a factor-of-IID process on $G$. 
Note that in this paper the state space $M$ is usually assumed to be finite. 

Such a process is said to be a \emph{block factor} if it can be obtained in a way that 
each (or, equivalently, any given) $X_v$ is influenced only by finitely many $Z_u$'s. 
In other words, there exists a finite radius $R$ such that $X_v$ is a function of $Z_u$, $u \in B_R(v)$, 
where $B_R(v)$ denotes the ball of radius $R$ around $v$ w.r.t.\ the graph distance in $G$. 
Block factors are also called \emph{finite-radius factors}. 
For a more detailed introduction to factors of IID, see \cite[Section 2]{entr_ineq}.

We will also need to assume \emph{unimodularity} for $G$, 
which is equivalent to the \emph{Mass-Transport Principle}, 
see Section \ref{sec:intro_fiid} for definitions and \cite{aldous_lyons} for details. 
\begin{theorem} \label{thm:block_factor_on_transitive}
Let $G$ be a unimodular transitive graph with vertex set $V$ and with degree $d$. 
Suppose that $X_v$, $v \in V$, is a block factor-of-IID process 
with radius $R$ and with a finite state space. 
For an arbitrary fixed vertex $o \in V$, let $B_R=B_R(o)$ denote the $R$-ball around $o$. 
Then we have 
$$ \frac{1}{2} \sum_{o' \sim o} H( X_o, X_{o'} ) \geq \beta H( X_o ) 
\mbox{, where } \beta = \frac{1}{2} \left( d + \min_{W \subseteq B_R} 
\frac{ | \partial W | }{|W|} \right) .$$
\end{theorem}
Theorem \ref{thm:fiid} stated in the introduction follows immediately 
considering the fact that any factor-of-IID process is the weak limit of block factors.  
A short argument proving this fact for transitive graphs 
can be found in \cite[Proof of Proposition 4.4]{harangi_virag}.

\begin{proof}[Proof of Theorem \ref{thm:block_factor_on_transitive}]
Let $Z_u$, $u \in V$, denote an IID process on the vertex set $V$ of $G$. 
Suppose that the process $X_v$, $v \in V$, is obtained as 
an $\Aut(G)$-equivariant function of the IID process 
in such a way that each $X_v$ is the function of $Z_u$, $u \in B_R(v)$. 

The proof will go along the same lines as in the finite setting. 
The first step is to define the value $h(u,v)$ for any pair of vertices $u,v$. 
Loosely speaking, it will denote the average information gained about $X_v$ 
when the seed $Z_u$ of vertex $u$ is revealed. 
We should specify the order in which we reveal the seeds. 
In the finite setting an arbitrary order could be chosen. 
We have to be more careful here since we will need that 
the function $h \colon V \times V \to [0,\infty)$ is \emph{diagonally $\Aut(G)$-invariant}, 
i.e.\ $h(\gamma u, \gamma v)=h(u,v)$ for all $\gamma \in \Aut(G)$ and $u,v \in V$. 
To this end when defining $h(u,v)$ we will take a random order of the vertices 
in the ball $B_R(v)$: a uniform random bijection $\pi \colon \{1, \ldots, N\} \to B_R(v)$, 
where $N$ denotes the number of vertices in $B_R(v)$. For a fixed bijection $\pi$ we define 
$$ h_\pi(u,v) \defeq 
H\left( X_v~|~Z_{\pi(1)},\ldots,Z_{\pi(i-1)} \right) - 
H\left( X_v~|~Z_{\pi(1)},\ldots,Z_{\pi(i)} \right) 
\mbox{, where } i=\pi^{-1}(u) .$$ 
We get $h(u,v)$ by taking expectation in $\pi$ : 
$$ h(u,v) \defeq \E_\pi h_\pi(u,v) .$$ 
This defines $h(u,v)$ whenever the distance of $u$ and $v$ is at most $R$. 
If the distance is larger than $R$, we set $h(u,v)=0$. 

The function $h$ is clearly diagonally $\Aut(G)$-invariant with the property that 
$$ H(X_o) = \sum_{u \in V} h(u,o) \mbox{ for any } o \in V ,$$
because the same holds for $h_\pi(u,o)$ for any fixed $\pi$. 
It is also easy to see that if $o$ and $o'$ are neighbors ($o' \sim o$ in notation), then 
\begin{equation} \label{eq:edge_bound} 
H(X_o,X_{o'}) \geq \sum_{u \in V} \max\left( h(u,o), h(u,o') \right) .
\end{equation}
Again, it suffices to see that the same holds for 
$h_\pi(u,o)$ and $h_{\pi'}(u,o')$, where $\pi$ and $\pi'$ are the ``restrictions'' 
of any fixed order of the vertices in $B_R(o) \cup B_R(o')$. 

Using \eqref{eq:edge_bound} and switching sums we get that for any fixed vertex $o$ 
\begin{equation} \label{eq:to_be_continued}
\sum_{o' \sim o} H( X_o, X_{o'} ) \geq 
\sum_{o' \sim o} \sum_{u \in V} \max\left( h(u,o), h(u,o') \right) = 
\sum_{u \in V} \underbrace{ \sum_{o' \sim o} \max\left( h(u,o), h(u,o') \right) }_{ g(u,o)\defeq } .
\end{equation}
Now we use the Mass-Transport Principle 
for the function $g(u,o)$ defined above. We obtain that the right-hand side of \eqref{eq:to_be_continued}  
is equal to the following for any fixed $u \in V$: 
\begin{equation} \label{eq:continued}
\sum_{o \in V} g(u,o) = 
\sum_{o \in V} \sum_{o' \sim o} \max\left( h(u,o), h(u,o') \right) = 
2 \sum_{(o,o')\in E(G)} \max\left( h(u,o), h(u,o') \right) .
\end{equation}
Since for fixed $u$ the function $h(u,\cdot)$ is supported on $B_R(u)$, 
we can use the Claim in the proof of the finite setting 
to get the following lower bound for \eqref{eq:continued}:
$$ 2\beta \sum_{o \in V} h(u,o) \mbox{, where } 
\beta = \min_{W \subseteq B_R(u)} 
\frac{\left| \overline{W} \right|}{|W|} = 
\min_{W \subseteq B_R(u)} \frac{d|W| + |\partial W|}{2|W|} .$$ 
Putting these together and applying the Mass-Transport Principle again, 
this time for the function $h$, we get the following for any fixed $o$:
$$ \sum_{o' \sim o} H( X_o, X_{o'} ) \geq 2\beta \sum_{u \in V} h(u,o) = 2\beta H( X_o ) ,$$
and the proof is complete. 
\end{proof}

\subsection{The general theorem}
As in the finite setting one may consider joint entropies of more than two vertices. 
Actually, at this point it is more natural to forget the graph structure completely. 
From this point on, $V$ will denote an arbitrary countable set with a left $\Ga$-action for some group $\Ga$. 
We will assume the action to be quasi-transitive (i.e.~having finitely many orbits) 
and to satisfy the following unimodularity-type condition. 
\begin{definition}
A function $f \colon V \times V \to [0,\infty]$ is said to be diagonally $\Ga$-invariant if 
$$ f(\ga u, \ga v) = f(u,v) \quad \forall u,v \in V \ \forall \ga \in \Ga .$$
We say that the \emph{Mass-Transport Principle} holds for a $\Ga$-action 
with respect to a measure $\mu$ on the orbit classes  
if for all diagonally $\Ga$-invariant function $f$ it holds that 
$$ \E_{\mu(o)} \sum_{v \in V} f(o,v) = \E_{\mu(o)} \sum_{v \in V} f(v,o) .$$ 
This is actually equivalent to the following statement on the orbit sizes of the stabilizers:  
$$ \frac{ |\Stab(o) o'| }{ |\Stab(o') o| } = \frac{ \mu(o') }{ \mu(o) } 
\mbox{ for any pair } o,o' .$$
\end{definition}
Next we introduce the weight function $\alpha \colon \PS(V) \to [0,\infty)$ 
that assigns nonnegative real numbers to subsets of $V$. We require $\alpha$ 
to satisfy the following properties: 
\begin{itemize}
\item $\alpha$ is $\Ga$-invariant: $\alpha_U = \alpha_{\gamma U}$ 
for any $U \subset V$ and $\ga \in \Ga$;
\item $\alpha$ is supported on finite sets: 
$\alpha_U = 0$ if $U$ is not finite;
\item $\alpha$ is ``locally finitely supported'': 
for any $v \in V$ there are finitely many $U$ 
such that $v \in U$ and $\alpha_U = 0$.
\end{itemize}
Now we are in a position to state the most general form of our theorem in the infinite setting. 
We started with the transitive graph version (Theorem \ref{thm:block_factor_on_transitive}) 
because both the statement and the proof are much easier to digest in that special case. 
The proof itself is a straightforward modification of that of Theorem \ref{thm:block_factor_on_transitive}, 
and we leave the details to the reader.
\begin{theorem} \label{thm:general} 
Let $V$, $\Ga$, $\mu$, $\alpha$ be as above, and let 
$Z_u$, $u \in V$, be an IID process on $G$. 
Suppose that $X_v$, $v \in V$, is a $\Ga$-factor of $Z$ with finite state space 
such that for any vertex $u$ the value of $Z_{u}$ influences only 
the labels $X_v$, $v \in W_u$, for some finite set $W_u \subset V$. 
Then 
$$ \E_{\mu(o)} \sum_{U : o \in U} \frac{\alpha_U}{|U|} H( X_U ) \geq \beta \E_{\mu(o)} H( X_o ) 
\mbox{, where } \beta = \min_{u \in V; W \subseteq W_u} 
\frac{ \sum_{U \subseteq V; U \cap W \neq \emptyset} \alpha_U }{|W|} .$$
\end{theorem}

\subsection{Corollaries}
We mention a couple of consequences of Theorem \ref{thm:general} in this section. 

The first special case concerns factors of IID over Cayley graphs of 
finitely-generated groups. 
\begin{theorem} \label{thm:cayley}
Let $\Ga$ be an infinite group with a finite symmetric generating set $S$. 
We denote the unit element of $\Ga$ by $e$. 
Suppose that $X_\ga$, $\ga \in \Ga$, is a $\Ga$-factor-of-IID process over 
(the Cayley diagram of) $\Ga$, where $X_\ga$ takes values from a finite set. 
Given arbitrary nonnegative real numbers 
$\alpha_\sigma \geq 0$, $\sigma \in S$, with $\alpha_\sigma  = \alpha_{\sigma^{-1}}$ 
we have the following inequality:
$$ \sum_{\sigma \in S} \alpha_\sigma H( X_e, X_\sigma ) \geq \beta H( X_e ) 
\mbox{, where } \beta = 
\inf_{W \subset \Ga, |W| < \infty} 
\frac{ \sum_{ \sigma \in S} \alpha_\sigma |W \cup \sigma W| }{|W|}.$$
\end{theorem}
\begin{remark}
When $\Ga$ is a free group and the weights are all equal, 
we get back an inequality that was known by Bowen \cite{bowen} 
as the fact the so-called \emph{$f$-invariant} is nonnegative 
for factors of the Bernoulli shift. See also \cite[Theorem 2.3]{mut_inf}, 
where this inequality is explicitly stated. 
\end{remark}
To see that Theorem \ref{thm:cayley} is indeed a special case of Theorem \ref{thm:general}, 
we set $V = \Ga$ and consider the natural left action of $\Ga$ on itself. 
It is a well-known fact that this is a unimodular transitive action. 
Using $\alpha_\sigma$ we define our weight function $\alpha \colon \PS(V) \to [0,\infty)$ as follows: 
for two-element sets of the form $U = \{ \ga, \ga \sigma\}$, $\ga \in \Ga, \sigma \in S$, 
we assign $\alpha_U \defeq \alpha_\sigma$, and we set $\alpha_U = 0$ for any other set $U$. 
This yields Theorem \ref{thm:cayley} for block factors. 
By taking weak limit, we get the inequality for any factor-of-IID process.

The other special case we briefly mention concerns factor-of-IID processes on $T_d$. 
The edge-vertex inequality \eqref{eq:edge-vertex} compares the entropy of a vertex 
to the joint entropy of two neighboring vertices. 
One could compare the vertex entropy to the entropy $H(X_{U_0})$ 
for any finite set $U_0 \subset V(T_d)$, 
and ask for the largest coefficient $c = c(U_0)$ such that 
$$ H(X_{U_0}) \geq c H(X_v) $$
holds for any factor-of-IID process $X$. 
Setting $\Gamma = \Aut(T_d)$, $V=V(T_d)$, 
and $\alpha(U) = 1$ whenever $U$ can be mapped to $U_0$ 
via an automorphism of $T_d$, we can apply Theorem \ref{thm:general} 
to get the above inequality with some coefficient $c$. 
It is actually not straightforward to see that the obtained coefficient is optimal. 
We omit the details here but it can be shown that 
the coefficient Theorem \ref{thm:general} provides is actually equal to 
$$ c= c(U_0) = \inf_{R} \frac{ |B_R(U_0)| }{ |B_R(v)| } 
\mbox{, where $B_R(U_0)$ is the $R$-neighborhood of $U_0$.} $$
It is easy to construct factor-of-IID processes for which 
the ratio $H(X_{U_0}) / H(X_v)$ tends to the above coefficient, 
see the example described in \cite[Section 5.1]{mut_inf}, 
showing that this coefficient is indeed the best possible.

\section{Beyond entropy: possible generalizations} \label{sec:beyond}

The goal of this section is to explore what properties of entropy were crucial in our arguments 
to see if there are any other quantities (assigned to random variables) that our proofs would work for. 

\subsection{Uncertainty functions}
Let us consider the space $\IM$ of probability measures on some measurable space $(A, \IA)$. 
This space is clearly closed under convex combination. Therefore we can talk about the convexity 
and concavity of $\IM \to \IR$ functions. 
\begin{definition}
We say that a function $\delta \colon \IM \to \IR$ is an \emph{uncertainty function} 
if $\delta$ is concave and $\delta(\nu)=0$ holds for any measure $\nu$ concentrated 
on one point (i.e.~unit mass).

The concavity of $\delta$ essentially means that the $\delta$-value of a mixture distribution must be 
greater than or equal to the expectation of the $\delta$-values. 
\end{definition}

To keep things simple, here we restrict our attention to the finite simple graph case: 
let $G$ be a finite graph with vertex set $V$ and edge set $E$. 
Suppose that $X_v$, $v \in V$, are $(A,\IA)$-valued random variables 
satisfying Assumption \ref{assumption}. This time, however, $X_v$ does not need to be discrete. 
Let $\delta$ be an uncertainty function for probability measures over $A$ as described above. 
Now we can talk about the uncertainty corresponding to a vertex: 
the $\delta$-value of (the distribution of) the random variable $X_v$. 
To be able to talk about the uncertainty corresponding to an edge $e=(u,v)$ as well, 
we need to choose an uncertainty function $\Delta$ for probability measures over $A \times A$. 
Of course, $\delta$ and $\Delta$ should be compatible in some sense. 
It turns out that the only assumption we need is the following.
\begin{assumption} \label{assumption2}
Let $\delta$ and $\Delta$ be uncertainty functions over $A$ and $A \times A$, respectively. 
Furthermore, for a measure $\mu$ on $A \times A$, 
let $\pi_1(\mu)$ and $\pi_2(\mu)$ denote the marginals of $\mu$. 
We assume that 
\begin{equation} \label{eq:delta}
\Delta - \delta \circ \pi_1 \mbox{ and } \Delta - \delta \circ \pi_2 
\mbox{ are both concave.} 
\end{equation}
\end{assumption}

Under this assumption, essentially the same proof works as in Theorem \ref{thm:weighted_hg}. 
For the sake of simplicity, let us assume that the random seeds $Z_1, \ldots, Z_m$ are discrete. 
Then for any event $Z_1=z_1, \ldots, Z_i=z_i$ with positive probability, 
we can consider the distribution of $X_v$ conditioned on this event and 
take the $\delta$-value of this conditional distribution. 
By taking expectation in $z_1, \ldots, z_i$ we get the expected uncertainty of the vertex $v$ 
after revealing the first $i$ seeds. We denote this by $\delta(i,v)$. 
Since $\delta$ is assumed to be concave, it is easy to see that 
$\delta(i,v)$ is monotone decreasing in $i$, and hence for any fixed $i$, 
$h(i,v) \defeq \delta(i-1,v)-\delta(i,v)$ is a nonnegative $V \to [0,\infty)$ function 
(supported on $W_i$).

Similarly, for an edge $e=(u,v)$, we define $\Delta(i,e)$ by the expectation 
of the $\Delta$-value of the joint distribution of $(X_u,X_v)$ conditioned on the first $i$ seeds. 
Then we define $h(i,e)$ as $\Delta(i-1,e) - \Delta(i,e)$. The crucial fact our proof relied on was that 
$$ h(i,e) \geq \max\left( h(i,u); h(i,v) \right) ,$$
which easily follows from \eqref{eq:delta}. 
The rest of the proof remains the same yielding the following result. 
\begin{theorem} \label{thm:uncertainty}
Suppose that $G=(V,E)$ is a finite graph, and 
$X_v$, $v \in V$, are $(A,\IA)$-valued random variables  
satisfying Assumption \ref{assumption}.  Furthermore, let $\delta$ and $\Delta$ be 
uncertainty functions as in Assumption \ref{assumption2}. Then 
\begin{equation*}
\sum_{e \in E} \Delta(X_e) \geq \beta \sum_{v \in V} \delta(X_v) 
\mbox{, where } \beta = \min_{1 \leq i \leq m; W \subseteq W_i} 
\frac{\left| \overline{W} \right|}{|W|} .
\end{equation*}
Recall that $\overline{W}$ is defined as $\{ e \in E : e \cap W \neq \emptyset \}$.
\end{theorem}

As for the infinite setting, Theorem \ref{thm:fiid} can be generalized similarly.

When $A$ is finite, we can use entropy for both $\delta$ and $\Delta$, 
and we get back our original inequality. Note that conditional entropy can be 
obtained as the expectation of entropies of conditional distributions; see the Appendix. 
Therefore $\delta(i,v)$ and $\Delta(i,e)$ in this case 
are the same conditional entropies that came up in the original proof.

\subsection{Correlation bound}
Another natural way to measure uncertainty of a distribution is variance. 
Next we study what Theorem \ref{thm:uncertainty} gives 
if we set $\delta(\nu)$ to be the variance of $\nu$. 

Suppose that $A=\IR$ equipped with the Borel $\sigma$-algebra. 
We will work under the assumption that the random variables 
$X_v$, $v \in V$, have finite variance. 
For a real-valued random variable $\nu$ of finite variance, 
$\delta(\nu)$ simply denotes the variance of $\nu$: 
$$ \delta(\nu) = \int_{\IR} x^2 \, \mathrm{d}\nu(x) - 
\left( \int_{\IR} x \, \mathrm{d}\nu(x) \right)^2 . $$
It is easy to see that this $\delta$ is indeed concave. 
(Note that the first term is linear in $\nu$ 
and hence does not influence convexity/concavity.)
 
A natural candidate for $\Delta$, the uncertainty function over $\IR \times \IR$, 
would be the covariance of the two coordinates. 
Some constant multiples of the variances of the marginals 
should be added to the covariance so that \eqref{eq:delta} holds true. 
More precisely, let $\mu$ be a probability measure on $\IR \times \IR$ 
with marginals $\mu_1 = \pi_1(\mu)$ and $\mu_2 = \pi_2(\mu)$. 
We want to find an uncertainty function of the following form:
$$ \Delta(\mu) = a \big( \delta(\mu_1) + \delta(\mu_2) \big) + 
b  \left( \int_{\IR \times \IR} xy \, \mathrm{d}\mu(x,y) - 
\left( \int_{\IR} x \, \mathrm{d}\mu_1(x) \right) \left( \int_{\IR} y \, \mathrm{d}\mu_2(y) \right) \right) .$$
It is easy to see that \eqref{eq:delta} is satisfied if and only if 
$$ a>1 \mbox{ and } |b| \leq 2\sqrt{a(a-1)} .$$
Setting $b = \pm 2\sqrt{a(a-1)}$, we get the following from Theorem \ref{thm:uncertainty} 
for $d$-regular graphs:
$$ \pm 2\sqrt{a(a-1)} \sum_{(u,v)\in E} \cov(X_u,X_v)
\leq d(a-\beta/d) \sum_{v \in V} \var(X_v) .$$ 
Easy calculations show that the best inequality is obtained 
by setting $a = \frac{\beta/d}{2\beta/d - 1}$, and we get the following result.
\begin{theorem}
Let $G$ be a $d$-regular graph with vertex set $V$ and edge set $E$. 
The random variables $X_v$, $v \in V$, and the coefficient $\beta$ are as in Theorem \ref{thm:regular}.
If $\var(X_v) < \infty$ for each $v \in V$, then 
$$ \left| \sum_{(u,v)\in E} \cov(X_u,X_v) \right| \leq \sqrt{\beta(d-\beta)} \sum_{v \in V} \var(X_v) .$$
\end{theorem}
This means that if each $X_v$ has the same variance, 
then the average correlation is at most $2\sqrt{\beta(d-\beta)}/d$. 
Note that the same bound could be deduced for factor-of-IID processes 
over unimodular transitive graphs with $\beta=(d+\psi)/2$ 
where $\psi$ is the edge Cheeger constant. 
In this infinite transitive graph setting, however, there is another way 
to prove this bound by combining the following known results. 
\begin{itemize}
\item The average correlation is bounded by the spectral radius of 
the adjacency operator divided by $d$. 
This follows from the results of \cite{spec} 
on the covariance structure of factor-of-IID processes.
\item By a Cheeger-tpye inequality \cite[Theorem 6.7]{lyons_peres_2017}, 
the spectral radius is bounded by 
$$\sqrt{d^2-\psi^2} = \sqrt{(d+\psi)(d-\psi)} = 2 \sqrt{\beta(d-\beta)}.$$  
\end{itemize} 

\subsection{Specific problems}
It is an intriguing question whether there are other general uncertainty functions 
that satisfy Assumption \ref{assumption2} and lead to useful inequalities. 
Instead of looking for further general inequalities, another approach is that, 
given a specific application, we try to find the best functions $\delta$ and $\Delta$ to use. 
In fact, there is hope that this approach might provide 
the optimal bound for factor-of-IID independent sets. 

In the setting of independent sets, each vertex can have two states ($0$, $1$)
and vertices of state $1$ cannot be neighbors. 
So the state space is $A = \{0,1\}$, and $\delta$ can be described 
by a concave function $f \colon [0,1] \to \IR$ with $f(0)=f(1)=0$. 
Furthermore, the measures $\mu$ in question are concentrated on 
$A \times A \setminus \{(1,1)\} = \{ (0,0); (0,1); (1,0) \}$. 
Such a measure $\mu$ can be represented by a point $(x,y)$ 
of the triangle $T = \{ (x,y) \, : \, x,y \geq 0;\ x+y \leq 1 \} \subset \IR^2$, 
where $x$, $y$, and $1-x-y$ are the probabilities of $(1,0)$, $(0,1)$, and $(0,0)$, respectively. 
Then $\Delta$ corresponds to a function $F \colon T \to \IR$ 
with $F(0,0)=F(0,1)=F(1,0)=0$ 
such that the functions $F(x,y) - f(x)$ and $F(x,y)-f(y)$ are both concave on $T$. 
For $d$-regular graphs, if each vertex is included in the independent set with probability $q$, 
then we get that $F(q,q) \geq (2\beta / d) f(q)$. 
In particular, we get the following for factor-of-IID independent sets in $T_d$, 
in which setting we have $\beta = d-1$. 
\begin{lemma}
Let $q$ be the (maximum) density of a factor-of-IID independent set 
over the $d$-regular tree $T_d$. By $T$ we denote the triangle 
$\{ (x,y) \, : \, x,y \geq 0;\ x+y \leq 1 \} \subset \IR^2$. 
Suppose that $f \colon [0,1] \to \IR$ is a concave function with $f(0)=f(1)=0$ 
and $F \colon T \to \IR$ is such that $F(0,0)=F(0,1)=F(1,0)=0$ 
and $F(x,y) - f(x)$ and $F(x,y)-f(y)$ are both concave on $T$. 
Then the following holds true: 
$$ F(q,q) \geq \frac{2(d-1)}{d} f(q) .$$
\end{lemma}

This gives the following upper bound for the maximum density $q$: 
$$ \inf \left\{ q \,:\, F(q,q) < (2(d-1) / d) f(q) \right\} .$$
So we need to find functions $f,F$ such that the above infimum is as small as possible. 
Solving this convex optimization problem, even numerically for a given $d$, would be very interesting. 
The question to determine the independence ratio of a random $d$-regular graph 
and the closely related problem of finding the maximum density of 
a factor-of-IID independent set on $T_d$ have been thoroughly studied; 
see Section \ref{sec:appl} for details.  
Note that if we use entropy for $\delta$ and $\Delta$, then we get the well-known Bollob\'as bound. 
In recent years there has been a lot of activity, especially regarding the case $d=3$: 
see \cite{E, hoppen_wormald, endreuj} for lower bounds 
and \cite{mckay,lelarge,balogh} for upper bounds. 
Finding the optimal $f$ and $F$ 
may lead to a breakthrough in this problem. 
The above optimization problem can also be written as an LP problem 
in terms of the second derivatives of $f$ and $F$. 
Its dual solution provides necessary and sufficient conditions for a factor-of-IID 
independent set to reach the obtained bound. Therefore the bound must be optimal 
provided that these conditions can be satisfied with an arbitrarily small error.

\section{Applications} \label{sec:appl}

Entropy inequalities can be used to acquire bounds for the size of 
different combinatorial structures (such as minimum or maximum cuts) on graphs. 
Bounds for independent sets have the most extensive literature. 
In this section we demonstrate how the new inequalities presented in this paper 
lead to results concerning random independent sets.

One of the first results in this direction is due to Bollob\'as 
who gave an upper bound for the independence ratio of random $d$-regular graphs 
\cite{bollobas}. 
The counting argument behind this result essentially proves 
the edge-vertex entropy inequality in this special setting. 
The upper bound obtained is asymptotically sharp: $2 \log(d) /d$ as $d \to \infty$. 
The lower bound is due to Frieze and {\L}uczak \cite{frieze_random_regular}.
Later Gamarnik and Sudan showed \cite{gamarnik} that so-called local algorithms 
cannot produce independent sets of this optimal size on large-girth graphs. 
Then in \cite{mustazeebalint} the edge-vertex entropy inequality was used 
in a non-trivial way to prove that independent sets provided by local algorithms 
can have density at most $\log(d) /d$ (asymptotically as $d \to \infty$). 
This is equivalent to saying that factor-of-IID independent sets on $T_d$ 
have density at most $( 1+o(1) ) \log(d) /d$. 
 
\subsection{Random independent sets and fractional colorings}
The graph parameter \emph{fractional chromatic number} can be defined as the inverse of 
the maximum ``homogeneous density'' of a random independent set. 
More precisely, let $G$ be a finite graph and let $q$ be the largest real number 
with the property that there exists a random independent set $I$ on $G$ 
such that any given vertex lies in $I$ with probability at least $q$. 
Then $q$ coincides with the inverse $1/\chi_{\textrm{\scriptsize f}}(G)$ of the fractional chromatic number. 

It is natural to consider random independent sets that are obtained by local algorithms; 
one could define the corresponding \emph{local fractional chromatic number}. 
We will treat the finite and infinite settings simultaneously. 

\textbf{Infinite setting:} given a unimodular transitive graph and a finite radius $R$, 
we consider radius-$R$ factor-of-IID processes $X$ such that $X_v$ takes values $0,1$ 
and $\{ v \, : \, X_v=1 \}$ is an independent set, that is, 
$X_u=X_v=1$ cannot happen for neighboring vertices $u,v$. 
The probability $q$ that $X_v = 1$ is the same for each $v$ due to invariance. 

\textbf{Finite setting:} given a $d$-regular finite graph $G$ 
and subsets $W_1, \ldots, W_m$ of the vertex set $V$, 
we consider random variables $X_v$, $v \in V$, as in Assumption \ref{assumption} 
such that $\{ v \, : \, X_v=1 \}$ is an independent set. 
Let $q$ be such that $X_v =1$ with probability at least $q$ for each vertex $v \in V$. 
One can easily modify such an algorithm in a way that 
each of these probabilities is actually equal to $q$.

\subsection{Bounds via entropy inequalities}
Our goal is to find an upper bound for $q$ in terms of the edge expansion of $G$. 
For each vertex $v$ we clearly have 
$$ H( X_v ) = - q\log q - (1-q)\log(1-q) .$$
Moreover, for any pair of neighboring vertices $u,v$: 
$$ H( X_u, X_v ) = - 2q\log q - (1-2q)\log(1-2q) .$$
Let 
$$ \varphi(q) = \frac{ H( X_u, X_v ) }{ H( X_v ) } - 1 = 
\frac{ - 2q\log q - (1-2q)\log(1-2q) }{ - q\log q - (1-q)\log(1-q) } - 1 .$$
It is easy to see that this function is continuous and monotone decreasing on $(0,1/2)$ 
with $\lim_{q \to 0+} \varphi(q) = 1$ and $\lim_{q \to 1/2-} \varphi(q) = 0$. 
Straightforward calculations show that the behavior of $\varphi$ around $0$ and $1/2$ is as follows:
\begin{align*}
\varphi(q) &= 1 + \frac{q}{\log q} + O\left( \frac{q}{(\log q)^2} \right) ;\\
\varphi(1/2 - t) &= \frac{-2t \log t}{\log 2} + O(t) .
\end{align*}
Therefore one can define a monotone decreasing inverse function 
$\varphi^{-1} \colon [0,1] \to [0,1/2]$ such that asymptotically at $0$ and $1$ we have  
\begin{align*}
\varphi^{-1}(x) - \frac{1}{2} & \sim \frac{\log 2}{2} \frac{x}{\log x} \mbox{ as } x \to 0;\\
\varphi^{-1}(1 - x) & \sim -x \log x  \mbox{ as } x \to 0.
\end{align*}
For an explicit upper bound near $0$ and $1$, one can use the following estimates for $\varphi^{-1}$:
$$
\varphi^{-1}(x) \leq
\begin{cases}
\frac{1}{2} - \frac{\log 2}{2} \frac{x}{-\log x + 2 \log( - \log x) } & 
\mbox{if } 0 < x < 1/30,\\
-(1-x)\log(1-x) & \mbox{if } 2/3 < x < 1.
\end{cases}
$$

Now let $\tau$ denote the edge expansion (w.r.t.\ the $R$-ball or $W_1, \ldots, W_m$) normalized by $d$:
$$  \tau = \min_{W \subseteq B_R} \frac{ | \partial W | }{d|W|} \mbox{ and }
 \tau = \min_{1 \leq i \leq m; W \subseteq W_i} \frac{ | \partial W | }{d|W|} $$
in the infinite and finite setting, respectively. Then our entropy inequalities 
(Theorem \ref{thm:regular} and Theorem \ref{thm:block_factor_on_transitive}) say that 
$$ \frac{ H(X_u, X_v) } {H(X_v)} - 1 \geq \tau .$$
It follows that $\varphi(q) \geq \tau$ must hold. Equivalently, we get the upper bound 
$$ q \leq \varphi^{-1}( \tau ) .$$
In the infinite setting we get Theorem \ref{thm:indset_bound} stated in the introduction. 
Corollary \ref{cor:tessellation} is concerned with the special case when our graph 
corresponds to a regular tessellation of the hyperbolic plane. The Cheeger constant of the graph 
$G_{d,k}$ is known \cite{haggstrom_jonasson_lyons,higuchi_tomoyuki} to be 
$$ (d-2) \sqrt{1-\frac{4}{(k-2)(d-2)}} .$$
If $k$ is fixed and $d \to \infty$, then we have $\tau = 1 - \big( 2+2/(k-2) \big) /d + o(1)$. 
Using the asymptotics of $\varphi^{-1}$ at $\tau=1$, Corollary \ref{cor:tessellation} follows. 

Furthermore, we can get something non-trivial even over amenable graphs. 
Let $G$ be the standard Cayley graph of $\IZ^n$ 
which is a $d$-regular transitive graph for $d=2n$. 
Although $G$ is amenable, the normalized edge expansion $\tau$ with respect to a 
finite-radius ball $B_R$ is at least $c n R^{-1}$ for some positive constant $c$, 
and hence Theorem \ref{thm:Zn} follows. (Note that one can easily 
prove---simply by using that labels of vertices at distance $2R+1$ are independent---the 
upper bound $1/2-c/R$ for any $n$. Therefore the bound given in Theorem \ref{thm:Zn} 
is interesting when $R$ is sub-exponential in $n$. Also note that it is easy to construct 
a factor-of-IID independent set on $\IZ^n$ with radius $R$ and density $1/2 - cn/R$.) 

Lastly, we mention one immediate corollary of the finite version. 
\begin{theorem} \label{thm:fin_cor}
Let $G$ be a random $d$-regular bipartite graph on $2N$ vertices. 
Then for any $\delta > 0$  there exists a positive $\eps$ such that 
the following holds true for $G$ with high probability as $N$ goes to infinity: 
a randomized algorithm on $G$ for which each seed influences at most $1-\delta$ fraction of the vertices 
cannot produce an independent set of homogeneous density at least $1/2 - \eps$ 
(that is, at least one vertex will be included in the independent set with probability less than $1/2 - \eps$). 
\end{theorem}

\section{Appendix} \label{sec:app}

\subsection*{Entropy and conditional entropy}
Let $X$ be a discrete random variable 
taking $m$ distinct values with probabilities $p_1, \ldots, p_m$. 
Then the \emph{Shannon entropy} of $X$ is defined as 
$$ H(X) \defeq \sum_{i=1}^m -p_i \log( p_i ) .$$
The \emph{joint entropy} of finitely many random variables 
is the entropy of their joint distribution. 

One can define the \emph{conditional entropy} of $X$ conditioned on $Y$ by 
$H(X|Y) = H(X,Y) - H(Y)$, where $H(X,Y)$ is the joint entropy of $X$ and $Y$. 
This conditional entropy can be expressed as as the expectation (in $Y$) 
of the entropy of the (conditional) distribution of $X$ conditioned on $Y$, that is, 
\begin{equation*} 
H(X | Y) = \sum_{j=1}^n \P(Y = y_j) 
\sum_{i=1}^m -\P(X=x_i \, | \, Y=y_j) \log \P(X=x_i \, | \, Y=y_j) ,
\end{equation*}
where $x_1, \ldots, x_m$ and $y_1, \ldots, y_n$ 
denote the values taken by $X$ and $Y$, respectively. 
In other words, if $f_i$ denotes the mapping $y \mapsto \P( X=x_i | Y=y)$, 
then 
\begin{equation*} 
H(X | Y) = \E \sum_{i=1}^m -f_i(Y) \log f_i(Y) .
\end{equation*}
This second definition of conditional entropy can be generalized even when 
the random variable is not discrete: for an event $A$ the mapping $y \mapsto \P( A | Y=y )$ 
needs to be replaced by the conditional expectation $\E( \ind_A | Y )$, 
which is a measurable function of $Y$. 
\begin{lemma} \label{lem:info_gained}
Let $X_1,X_2,Y_1,Y_2$ be random variables 
such that $X_1$ and $X_2$ take finitely many values. Then 
$$ H( X_1,X_2 | Y_1 ) - H( X_1,X_2 | Y_1,Y_2 ) 
\geq H( X_1 | Y_1 ) - H( X_1 | Y_1,Y_2 ) .$$
\end{lemma}
\begin{proof}
We will refer to the left-hand side and the right-hand side of the above inequality 
as $LHS$ and $RHS$, respectively. Let us notice that 
\begin{align*}
H(X_1,X_2 | Y_1) &= H(X_1 | Y_1) + H(X_2 | X_1,Y_1) ;\\
H(X_1,X_2 | Y_1,Y_2) &= H(X_1 | Y_1,Y_2) + H(X_2 | X_1,Y_1,Y_2) .
\end{align*}
By taking the difference of the above equations we get 
$$ LHS = RHS + 
\underbrace{ H(X_2 | X_1,Y_1) - H(X_2 | X_1,Y_1,Y_2) }_{\geq 0} 
\geq RHS .$$
Note that the above argument works even if 
$Y_1$ and $Y_2$ are not necessarily discrete.  
\end{proof}

Next we state a weighted version of Shearer's inequality 
and we include a short proof. 
\begin{theorem}[Shearer's inequality, weighted version]
Let $X_v$, $v \in V$, be finitely many random variables. 
Suppose that for each $U \subseteq V$ a nonnegative weight $\alpha_U$ is given. Then 
$$ \sum_{U \subseteq V} \alpha_U H( X_U ) \geq \lambda H( X_V ) \mbox{, where } 
\lambda = \min_{v \in V} \sum_{v \in U \subseteq V} \alpha_U .$$
\end{theorem}
\begin{proof} 
We may assume that $V = \{1, \ldots, n\}$ for some positive integer $n$. 
For a sequence $1 \leq k_1 < \ldots < k_r \leq n$ we decompose 
the joint entropy $H(X_U)$ corresponding to the subset $U = \{k_1, \ldots, k_r\}$ 
in the following way:
$$ H(X_U) = \sum_{i=1}^r H\left( X_{k_i}~|~X_{k_1}, \ldots, X_{k_{i-1}} \right) \geq 
\sum_{i=1}^r H\left( X_{k_i}~|~X_1, \ldots, X_{k_i-1} \right). $$
It follows that 
$$ \sum_{U \subseteq V} \alpha_U H( X_U ) \geq \sum_{k=1}^n 
\underbrace{ \sum_{k \in U \subseteq V} \alpha_U }_{\geq \lambda} 
H\left( X_k~|~X_1, \ldots, X_{k-1} \right) \geq \lambda H( X_V ) .$$
\end{proof}

\subsection*{Acknowledgements}
The authors would like to thank Lewis Bowen, David Gamarnik, Russell Lyons, and Mustazee Rahman 
for their helpful feedback and remarks. The authors are also grateful to an anonymous referee 
for a very careful reading of the manuscript and for several useful comments.

Endre Cs\'oka was supported 
by Marie Sk{\l}odowska-Curie Individual Fellowship grant no.\ 750857 
and partially supported by ERC Consolidator Grant InvGroGra 648017. 
B\'alint Vir\'ag and Viktor Harangi were supported by 
``MTA R\'enyi Lend\"ulet V\'eletlen Spektrum Kutat\'ocsoport''. 
B\'alint Vir\'ag was also supported by the Canada Research Chair program, 
the NSERC Discovery Accelerator grant, 
and the ERC Consolidator Grant InvGroGra 648017.

\bibliographystyle{abbrv}
\bibliography{refs}

\end{document}